\documentclass{amsart}
\usepackage[all]{xy}
\usepackage{amsmath}
\usepackage{amssymb}
\usepackage{amsthm}
\usepackage{amscd}

\newcommand{\R}{\mathcal{R}}
\newcommand{\M}{\mathcal{M}}
\newcommand{\n}{\mathcal{N}}

\newcommand{\F}{\mathcal{F}}

\newcommand{\Z}{\Bbb Z}
\newcommand{\I}{\Bbb I}
\newcommand{\RR}{\Bbb R}

\newcommand{\C}{\Bbb C}

\newtheorem{theorem}{Theorem}

\newtheorem{proposition}[theorem]{Proposition}
\newtheorem{cor}[theorem]{Corollary}
\newtheorem{lemma}[theorem]{Lemma}

\newtheorem{remark}[theorem]{Remark}

\DeclareMathOperator{\N}{\mathbb{N}}

\numberwithin{theorem}{section}

\begin{document}
\title[Spherical Means in the Hypercube]{Dimension-Free $L^p$-Maximal Inequalities for Spherical Means in the Hypercube}
\author{Ben Krause}
\address{
Department of Mathematics
The University of British Columbia \\
1984 Mathematics Road
Vancouver, B.C.
Canada V6T 1Z2}
\email{benkrause@math.ubc.ca}
\date{\today}
\maketitle

\section{Abstract}
We extend the main result of Harrow, Kolla, and Schulman  -- the existence of dimension-free $L^2$-bounds for the spherical maximal function in the hypercube, $\{0,1\}^N$ -- to all $L^p, p > 1$. Our approach is motivated by the spectral technique developed by Nevo and Stein, and by Stein, in the context of pointwise ergodic theorems on general groups. We provide an example which demonstrates that no dimension-free weak-type $1-1$ bound exists at the endpoint.

\section{Introduction}

Let $\mathbb{I}^N = \{ 0, 1\}^N$ denote the $N$-dimensional hypercube equipped with Hamming metric,
\[ |y|= \left|\big( y(1), \dots, y(N) \big) \right|:= \#\{ 1 \leq i \leq N : y(i) = 1 \};\]
for functions on $\mathbb{I}^N$, 
\[ f : \mathbb{I}^N \to \mathbb{C},\]
we define the $L^p$ norms, $1 \leq p < \infty$
\[ \| f \|_{L^p(\I^N)} := \left( \sum_{x \in \I^N} |f(x)|^p \right)^{1/p},\]
and 
\[ \|f\|_{L^\infty(\I^N)} := \max_{x \in \I^N} |f(x)|.\]
Recall that the dual group of $\mathbb{I}^N$ is itself $\mathbb{I}^N$; we will use lower case letters (e.g.\ $x,y,\dots$) to denote points in the group, while 
capital letters (e.g.\ $S, E, \dots$) will be used to denote frequencies in the dual group; for such frequencies, we define the ($L^2$-normalized) characters
\[ \chi_S(x):= \frac{1}{2^{N/2}} (-1)^{x \cdot S} := \frac{1}{2^{N/2}} (-1)^{\sum_{i=1}^N x(i) S(i)} =
\frac{1}{2^{N/2}} \prod_{i : S(i) = 1} (-1)^{x(i)}; \]
and the Fourier transform
\[ \F f(S) = \hat{f}(S) = \sum_{x \in \I^N} f(x) \chi_S(x);\]
we remark that on $\I^N$, Fourier transform and inverse Fourier transform coincide, i.e.\ $\F \F$ acts as the identity operator.

For $0 \leq k \leq N$, let
\[\sigma_k := \frac{1}{|\{ |x| = k \} |} \mathbf{1}_{ \{ |x| = k\}} = \frac{1}{\binom{N}{k}} \mathbf{1}_{\{|x| = k\}},
\]
denote the $L^1$-normalized indicator function of the $k$-sphere. The main goal of this paper will be to establish \emph{dimension-independent} $L^p \to L^p$ estimates on the following maximal function,
\[ f^{**}(x):= \sup_{0 \leq k \leq N} |\sigma_k*f|,\]
where convolution is defined
\[ \sigma_k*f(x) := \sum_{y \in \I^N} f(x-y) \sigma_k(y);\]
note that we are treating $\I$ as the field with two elements, so subtraction (equivalently, addition) is performed componentwise $\mod 2$. We call attention to the fact that convolution with the $\{ \sigma_k\}$ are (positive) $L^1$ and $L^\infty$ contractions, and thus are $L^p$ contractions for any $1 \leq p \leq \infty$:
\begin{equation}\label{e:contr}
\| \sigma_k *f \|_{L^p(\mathbb{I}^N)} \leq \| f \|_{L^p(\mathbb{I}^N)}
\end{equation}
for any $0 \leq k \leq N$ and any $1 \leq p \leq \infty$.
Now, since we are studying a maximal function taken over positive operators, there is no loss of generality in assuming that each function considered in this paper is non-negative. Furthermore, setting
\[ f^{*}(x):= \sup_{0 \leq k \leq N/2} |\sigma_k*f|,\]
and noting that $\sigma_{N-k} = \sigma_N * \sigma_k$, we may estimate
\[ f^{**} \leq f^* + (\sigma_N*f)^*;\]
for the purposes of establishing $L^p$-estimates on $f^{**}$, it is therefore enough to consider the maximal function $f^*$ (cf.\ \eqref{e:contr}). It is this operator which we proceed to analyze.

The motivation for our work comes from a recent paper of Harrow, Kolla, and Schulman \cite{HKS}, who established the following two results:

\begin{proposition}[cf. \cite{HKS} Lemmas 10-11]\label{smooth}
The smooth maximal function,
\[ \M_{smooth} f := \sup_{0 \leq K \leq N/2} \frac{1}{K+1} | \sum_{n \leq K} \sigma_n *f | \]
satisfies dimension independent weak-type $(1,1)$ inequalities. In particular, there exists an absolute constant $A_{1,1}$ \emph{independent of $N$} so that for each $\lambda \geq 0$
\[ \lambda \cdot |\{ x \in \I^N : \M_{smooth} f > \lambda \}| \leq A_{1,1} \| f \|_{L^1(\I^N)}\]
for each $f: \I^N \to \C$.
\end{proposition}
\begin{remark}
By Marcinkiewicz interpolation (cf.\ e.g.\ \cite{Z}) against the trivial $L^\infty$ bound, this result implies the existence of dimension-independent $L^p$ bounds on $\M_{smooth}$: for each $p >1$ there exist dimension-independent constants, $A_p$, so that
\[ \| \M_{smooth} f \|_{L^p(\I^N)} \leq A_p \|f\|_{L^p(\I^N)}.\]
\end{remark}

This proposition proved a key ingredient in establishing the main result of \cite{HKS}:
\begin{theorem}
There exists a constant $C_2>0$ so that for all $N$
\[ \| f^* \|_{L^2(\I^N)} \leq C_2 \|f\|_{L^2(\I^N)}\]
for each $f : \I^N \to \C$.
\end{theorem}
\begin{remark}
Once again, by interpolation this implies the existence of dimension-independent constants $C_p >0$ for $p \geq 2$ so that
\[ \|f^* \|_{L^p(\I^N)} \leq C_p \|f\|_{L^p(\I^N)} \]
for each $f : \I^N \to \C$.
\end{remark}

Their argument is an elegant application of Stein's method \cite{S}, used in extending the well-known Hopf-Dunford-Schwartz maximal theorem for semi-groups to more ``singular'' maximal averages.

The argument of \cite{HKS} breaks into two main steps:

\begin{enumerate}
\item One first establishes Proposition \ref{smooth} by comparing the operator $\M_{smooth}$ with the maximal function
\[ \sup_{0 < T < \infty} \frac{1}{T} \int_0^T \n_t f(x) \ dt,\]
where $\n_t$ is the noise semi-group from Boolean Analysis \cite[\S 4]{HKS}, 
\[ \n_t f(x) := \sum_{S \in \I^N} e^{-t|S|} \hat{f}(S) \chi_S(x), \ 0 < t < \infty;\]
\item The ``rougher'' maximal function $f^*$ is compared to the ``smoother'' maximal function in $L^2$ by using Littlewood-Paley theory on the group $\I^N$. The key tool is an analysis of the (radial) spherical multipliers
\[ \F \sigma_k(S) := \kappa_{k}^N(S) ,\]
the \emph{Krawtchouk} polynomials, which are introduced and discussed in \cite[\S 2]{HKS}.
\end{enumerate}

We continue the analysis of the Krawtchouk polynomials to the extent that we are able to bring the more general comparison technique of \cite{S}, \cite{NS} to bear. For a Euclidean analogue of this technique used in a similar study of ``rough'' maximal functions, we refer the reader to e.g.\ \cite[\S XI, 2]{S1}.

Our main result is the below theorem.

\begin{theorem}\label{big}
For any $p > 1$, there exist absolute, dimension-independent constants $C_p >0$ so that
\[ \| f^* \|_{L^p(\I^N)} \leq C_p \| f \|_{L^p(\I^N)} \]
for each $f : \I^N \to \C$.
\end{theorem}

This leads directly to the following corollary (cf. \cite[\S 1.1]{HKS}). In what follows, $\epsilon_p$ will denote sufficiently small numbers, depending only on the $L^p \to L^p$ operator norm bound on $f^*$.

\begin{cor}
Let $p > 1$ be arbitrary. Suppose that $L \subset \I^N$ is a subset of the hypercube that has sufficiently small relative density $\epsilon_L \leq \epsilon_p$:
\[ \frac{|L|}{2^N} = \epsilon_L.\]
Then there exists some $z = z_L \in \I^N$ so that the for \emph{every} $0 \leq k \leq N$, the fraction of the $k$-sphere centered at $z_L$, $\{ x \in \I^N : |x-z_L| = k \}$, which intersects $L$ is no more than a constant multiple of $\epsilon_L^{1/p}$.
\end{cor}
\begin{proof}
One averages
\[ \frac{1}{2^N} \sum_{x \in \I^N} |(\mathbf{1}_L)^{**}|^p(x) \leq C_p^p \frac{|L|}{2^N} = C_p^p \epsilon_L.\]
By the pigeon-hole principle, there exists a $z_L \in \I^N$ so that 
\[ (\mathbf{1}_L)^{**}(z_L) \leq C_p \epsilon_L^{1/p},\]
which is non-vacuous for $\epsilon_L$ sufficiently small. In particular,
\[ \sigma_k*\mathbf{1}_L(z_L) = \frac{|\{x \in \I^N: |x-z_L| = k \}|}{|\{ x \in \I^N : |x| =k\}|} \leq C_p \epsilon_L^{1/p} \]
for each $k$, which is the desired conclusion.
\end{proof}

Unfortunately, our argument, which relies on semigroup techniques, breaks down at the endpoint $L^1 \to L^{1,\infty}$.
Indeed, as shown by Ornstein \cite{O}, it is not in general possible to convert $L^p$-semigroup estimates to weak-type $(1,1)$ bounds.
Away from the semigroup setting, the problem of converting $L^p, p>1$ bounds to weak-type $(1,1)$ estimates remains not just formally more difficult, but often impossible. A familiar example occurs in the Euclidean setting, where dimension-free estimates are proved for the Hardy-Littlewood maximal function for cubes inside $\RR^d$ \cite{B},
\[ \M_{cubes}f(x) := \sup_{r>0} \frac{1}{r^d} \int_{[-r/2,r/2]^d} |f(x-y)| \ dy,\]
in contrast to the fact that the weak-type $(1,1)$ norm of $\M_{cubes}$ must grow with dimension, see \cite{Al} and \cite{Au}. We refer the reader to \cite{NT} for further examples and discussion.

In the other direction, 
dimension-independent weak-type $(1,1)$ estimates have been established on the free-group in \cite{RT}, and later in \cite[\S 5]{NT};
both arguments were driven by the underlying geometry of the group at hand: a strong isoperimetric inequality, and uniqueness of geodesics, anchor the respective proofs. For a further discussion of the connections between group geometry and weak-type bounds, we refer the reader to \cite{NT}.

Given the dimension-independent weak-type $(1,1)$ boundedness of $\M_{smooth}$, one might hope that the spherical maximal function might itself satisfy similar dimension-independent bounds; to the extent that the geometry of the hypercube renders ineffective both of the above techniques, however, it is perhaps unsurprising that obtaining a dimension-independent weak-type estimate is not possible:

By testing against, $\delta = \mathbf{1}_{\{(0,\dots,0)\}}$, the point mass at the origin (\S 4), we show that the weak-type $(1,1)$ operator norm of the spherical maximal function, $C_1 = C_1(N)$, must grow at least like $\sqrt{N}$. In particular, there exists an absolute constant $c>0$ so that
\[ \sup_{ f \neq 0} \frac{\| f^{**} \|_{L^{1,\infty}(\I^N)}}{ \| f \|_{L^1(\I^N)}} \geq
\frac{\| \delta^{**} \|_{L^{1,\infty}(\I^N)}}{ \| \delta \|_{L^1(\I^N)}} \geq
c \sqrt{N}. \]
Here, the supremum is taken over all non-zero functions $f : \I^N \to \C$.

\begin{remark}
As is shown in \cite[\S 4]{HKS}, the maximal function associated to the above-mentioned noise semigroup, $\n_*f := \sup_t | \n_t f|$,   pointwise dominates the dampened maximal function
\[ \frac{1}{\sqrt{N}} f^*. \]
Thus, the problem of determining whether the $\sqrt{N}$ growth of the weak-type $(1,1)$ operator norm is sharp could reduce to determining whether $\n_*$ is bounded independent of dimension; this seems like a very challenging problem, as in general, such semi-group maximal functions are not bounded at the $L^1 \to L^{1,\infty}$ endpoint (cf.\ \cite{O}).
We look forward to pursuing this line of inquiry in further research.
\end{remark}

\subsection{Acknowledgements}
The author would like to thank Michael Lacey for bringing this problem to his attention, Alexander Bufetov for helpful conversations,
Igor Pak and Ryan O'Donnell for their encouragement, and his advisor, Terence Tao, for his continued support. Finally, the author wishes to thank the anonymous referee for his great input in correcting a significant error in a previous draft, and for all his effort in helping to improve this paper.

\subsection{Notation}
We will make use of the modified Vinogradov notation. We use $X \lesssim Y$, or $Y \gtrsim X$ to denote the estimate $X \leq CY$ for an absolute constant $C$. If we need $C$ to depend on a parameter, we shall indicate this by subscripts, thus for instance $X \lesssim_m Y$ denotes the estimate $X \leq C_m Y$ for some $C_m$ depending on $m$. We use $X \approx Y$ as shorthand for $X \lesssim Y \lesssim X$, and similarly for $X \approx_m Y$.

\section{Proof of Theorem \ref{big}}
We break the argument into subsections. Throughout, all parameters $k,r, \dots$ will be non-negative.

\subsection{Initial Reductions}
We begin by defining the auxiliary maximal function,
\[ \M f(x):= \sup_{2k \leq M} |\sigma_{2k} f|(x) 
\]
where $M:= \lceil N/2 \rceil$ is the least integer greater than or equal to $N/2$.

Up to the identity
\begin{equation}\label{e:o2e}
\sigma_1 * \sigma_k = \frac{k}{N} \sigma_{k-1} + \frac{N-k}{N} \sigma_{k+1},
\end{equation}
which we will establish below, it will suffice to study the operator $\M$. First, using \eqref{e:o2e}, we obtain the pointwise bound:
\[ \aligned
&\sigma_1 * \sigma_{2n-2} + \sigma_1 * \sigma_{2n} \\
& \qquad \qquad \qquad = \frac{2n-2}{N} \sigma_{2n-3}+
\frac{N+2}{N} \sigma_{2n-1} +
\frac{N-2n}{N} \sigma_{2n+1} \\
& \qquad \qquad \qquad \geq \sigma_{2n-1}. \endaligned \]
But now, for any odd radius, $2n-1$, we simply majorize
\[ \sigma_{2n-1}*f \leq \sigma_{2n-2} *( \sigma_1*f) + \sigma_{2n} * (\sigma_1*f) \leq 2\M(\sigma_1*f),\]
and take into account \eqref{e:contr}.

To establish \eqref{e:o2e} we compute:
\[ \aligned
\sigma_1 * \sigma_k(x) &= \sum_{y \in \mathbb{I}^N} \sigma_k(x-y) \sigma_1(y) \\
&= \frac{1}{\binom{N}{k}} \cdot \frac{1}{N} |\{|y| = 1 : |x-y| = k\}| \\
&= \frac{1}{\binom{N}{k}} \cdot \frac{1}{N} \left((N-k+1) \cdot \mathbf{1}_{\{ |x|=k-1 \} } + (k+1) \cdot \mathbf{1}_{ \{ |x| = k+1 \} }  \right) \\
&= \frac{k}{N} \sigma_{k-1} + \frac{N-k}{N} \sigma_{k+1},
\endaligned \]
where in the second last line, we used that if $|x|=k-1$, there are exactly
\[ N - (k-1) = N-k+1\]
possible points $|y|=1$ so that
\[ |x-y| = k,\]
and similarly there are $k+1$ points on $\{|y|=1\}$ so that if $|x| = k+1$, $|x-y| = k$.

\medskip

We now turn to:
\subsection{Krawtchouk Polynomials and Spectral Preliminaries}
Following the lead and notation of \cite[\S 2.2]{NS}, for a sequence of numbers $\{ u_k\}_{k=0}^\infty$, we define the discrete differentiation operators,
\[ \aligned
\triangle^0 u_k &:= u_k \\
\triangle^1 u_k &:= u_k - u_{k-1}, \ \triangle u_0 := u_0\\
&\ \vdots \\
\triangle^m u_k &:= \triangle (\triangle^{m-1} u_k) =
\sum_{j=0}^m (-1)^j \binom{m}{j} u_{k-j}, \\
&\ \vdots \endaligned \]
We will let these operators act on the sequence of functions 
\[ \{ \sigma_{2k}(x), \ k : 0 \leq 2k \leq M \},\] so that e.g.\
\begin{equation}\label{e:tri}
\triangle^m \sigma_{2k}(x) = \sum_{j=0}^m (-1)^j\binom{m}{j} \sigma_{2(k-j)}(x).
\end{equation}

We will need to consider the associated (radial) multipliers
\[ \F (\triangle^j \sigma_{2k}) (S);\]
again, we 
we shall follow the lead of \cite[\S 3.2]{HKS}.

With $|S| = r$, we have
\begin{equation}\label{e:mult} \F \sigma_{2k}(S) = \sum_{j=0}^{2k} (-1)^j \frac{ \binom{r}{j} \cdot \binom{N-r}{2k-j} }{ \binom{N}{ 2k} }.
\end{equation}
Indeed, with 
\[ E^N_r:=  (\underbrace{1, \dots, 1}_{\text{$r$ ones}}, \underbrace{0 \dots 0}_{\text{$N-r$ zeros}}) \]
denoting the element with $r$ ``1''s followed by $N-r$ ``0''s, using the radiality of $\F \sigma_{2k}$ we compute
\[ \aligned
\F \sigma_{2k} (S) &= \F \sigma_{2k} (E^N_r)\\
&= \frac{1}{\binom{N}{2k}} \sum_{|y| = 2k} (-1)^{y \cdot E^N_r} \\
&= \frac{1}{\binom{N}{2k}} \sum_{j=0}^{2k} \sum_{|y| = 2k, y \cdot E^N_r = j} (-1)^{y \cdot E^N_r}\\
&= \frac{1}{\binom{N}{2k}} \sum_{j=0}^{2k} (-1)^j \binom{r}{j} \cdot \binom{N-r}{2k-j}. \endaligned\]
The expression on the right is the normalized $(2k)$th Krawtchouk Polynomial, $\kappa_{2k}^{N}(r)$. In particular, as remarked in \cite[p.8]{HKS},
\[ \binom{N}{k} \kappa_k^N(r) \]
counts the $k$-element subsets of $\{1,\dots,N\}$ according to the parity of intersection with the set $\{1,\dots,r\}$.

We collect the properties of these multipliers relevant to our analysis in the following lemma. For a fuller discussion of Krawtchouk Polynomials, we refer the reader to \cite{K} or to \cite{L}.

\begin{lemma}\label{KR}
The Krawtchouk polynomials $\kappa_k^N(r)$ satisfy the symmetries
\[ \kappa_k^N(r) = \kappa_r^N(k) = (-1)^k \cdot \kappa_k^N(N-r). \]
Moreover, there exists some $c > 0$ so that for all $0 \leq k,r \leq M$, we have the bound
\[ |\kappa_k^N (r)| \leq e^{-c \cdot \frac{kr}{N} }. \]
\end{lemma}
\begin{proof}
The symmetries are immediate upon inspection, and the quantitative bound appears as \cite[Lemma 7]{HKS}.
\end{proof}

\begin{lemma}\label{diff}
For any $r,l\leq N$,
\[ \aligned
\kappa_r^N(l) + \kappa_r^N(l-1) &= 2 \frac{N-r}{N} \cdot \kappa_r^{N-1}(l-1) \ \text{ and} \\
\kappa_r^N(l) - \kappa_r^N(l-1) &= -2 \frac{r}{N} \cdot \kappa_{r-1}^{N-1}(l-1). \endaligned \]
Consequently,
\[ \kappa_r^N(l) - \kappa_r^N(l-2) = -4 \frac{r(N-r)}{N(N-1)} \cdot \kappa_{r-1}^{N-2}(l-2) = -4 \frac{\binom{N-2}{r-1}}{\binom{N}{r}} \cdot \kappa_{r-1}^{N-2}(l-2).\]
\end{lemma}
\begin{proof}
This follows from conditioning on whether $r$-element subsets contain the element $l$ \cite[\S 3]{HKS}.

We provide details in the case of addition; the subtraction follows a similar line of reasoning, and at any rate appears as \cite[Lemma 3.2]{GKKS}.

One computes
\[ \aligned
\kappa_r^N(l) + \kappa_r^N(l-1) &=
\sum_{x \in \I^N} \sigma_r(x) (-1)^{x \cdot E_l^N} + \sum_{x \in \I^N} \sigma_r(x) (-1)^{x \cdot E_{l-1}^N} \\
&= \frac{1}{\binom{N}{r}} \sum_{x \in \I^N : |x| = r} (-1)^{x(1) + \dots + x(l-1)} \cdot ( (-1)^{x(l)} + 1 ) \\
&= \frac{2}{\binom{N}{r}} \sum_{x \in \I^N : |x| = r, x(l) = 0} (-1)^{x(1) + \dots + x(l-1)}. \endaligned\]
But, this sum can be expressed as
\[ \aligned
\frac{2}{\binom{N}{r}} \sum_{y \in \I^{N-1} : |y| = r} (-1)^{y(1) + \dots + y(l-1)} &=
\frac{2\binom{N-1}{r}}{\binom{N}{r}} \frac{1}{\binom{N-1}{r}} \sum_{y \in \I^{N-1} : |y| = r} (-1)^{y \cdot E_{l-1}^{N-1}} \\
&= 2 \frac{N-r}{N} \cdot \kappa_r^{N-1}(l-1), \endaligned\]
as desired.
\end{proof}

It will be useful to \emph{define}
\[ \aligned 
\partial^0 \kappa_r^N(l) &:= \kappa_r^N(l), \\
\partial \kappa_r^N(l) &:= \partial^1 \kappa_r^N(l) := \kappa_r^N(l) - \kappa_r^N(l-2) \\
&\ \vdots \\
\partial^m \kappa_r^N(l) &:= \partial ( \partial^{m-1} \kappa_r^N(l))\\
&\ \vdots
\endaligned \]
provided $m \leq \min\{ r, l/2\}$. Otherwise we set $\partial^m \kappa_r^N(l) :=0$.

By repeated applications of Lemma \ref{diff}, we see that
\[ \partial^m \kappa_r^N(l) = (-4)^m \frac{\binom{N-2m}{r-m}}{\binom{N}{r}} \kappa_{r-m}^{N-2m}(l-2m),\]
where $\binom{N-2m}{r-m} := 0$ if $r-m > N-2m$.

Now, using the symmetry $\kappa^N_k(r) = \kappa^N_r(k)$, we find that for $2k \leq M$, and any frequency $|S| = r$,
\[ \F \left( \triangle^m \sigma_{2k} \right) (S) = \partial^m \kappa_r^N(2k) = (-4)^m \frac{\binom{N-2m}{r-m}}{\binom{N}{r}} \kappa_{r-m}^{N-2m}(2(k-m)) \]
provided $m \leq \min\{ r, k\}$.
For future reference, we remark that for $r \leq M$, the quantitative estimate from Lemma \ref{KR} is effective,
\begin{equation}\label{e:smallr}
|\kappa_{r-m}^{N-2m}(2(k-m))| \leq e^{-c \frac{(r-m)(2k -2m)}{N-2m}}
\end{equation}
for an appropriate constant $c$, since in this case 
\[ r-m, \ 2k-2m \leq \frac{N-2m}{2};\] for $r > M$, we twice use the symmetry of Lemma \ref{KR} to bound
\begin{equation}\label{e:biggr}
|\kappa_{r-m}^{N-2m}(2(k-m))| = 
|\kappa_{N-r-m}^{N-2m}(2(k-m))|
\leq e^{-c \frac{(N-r-m)(2k -2m)}{N-2m}}
\end{equation}


\subsection{A Review of Nevo-Stein}
In this subsection, we will review the comparison argument of \cite{S} as it relates to our current setting. For a fuller treatment, we refer the reader to \cite{NS}.

Since the convolution operators with kernels $\{ \sigma_{2k} \}$ are self-adjoint, positive, norm-one $L^1$- and $L^\infty$-contractions \eqref{e:contr}, 
we may use the following outline from \cite{S}, \cite{NS}:

With $\alpha, \beta \in \RR$, $\lambda = \alpha + i \beta \in \C$, we recall the complex binomial coefficients
\begin{equation}\label{e:binom} A^\lambda_n := \frac{ (\lambda + 1)(\lambda + 2)\cdot \ldots \cdot (\lambda + n)}{n!}, \ A_0^\lambda:= 1, A_{-1}^\lambda:=0.
\end{equation}
We define the Cesaro means
\begin{equation}\label{e:ces} S_n^\lambda f(x) := \sum_{k=0}^{n} A_{n-k}^\lambda \sigma_{2k}* f(x), \ \lambda \in \C,
\end{equation}
for $n \leq M/2$
and remark that in the special case that $\lambda= -m-1$ is a negative integer, we have
\begin{equation}\label{e:goodces}
S_n^{-m-1}f(x) = \sum_{k=0}^{n} \triangle^m \sigma_{2k}* f(x)
\end{equation}
\cite[p. 143]{NS}. In particular, when $m > M/2$, $S_n^{-m-1} f \equiv 0$. For future reference, we also observe that for all $\lambda \in \mathbb{C}$,
\begin{equation}\label{e:id}
S_0^\lambda f = A_0^\lambda \sigma_0*f = \sigma_0*f = f.
\end{equation}

The maximal functions associated to these higher Cesaro means are
\begin{equation}\label{e:max} S_*^\lambda f(x):= \max_{0 \leq n \leq N/4} \left| \frac{S_n^\lambda f(x)}{(n+1)^{\lambda +1}} \right|. \end{equation}

The following lemmas are finitary adaptations of the results in \cite{NS}; we emphasize that the formal nature of the arguments in \cite{NS} allows them to be applied in much greater generality than our current setting.

\medskip

\begin{lemma}[cf. \cite{NS}, Proof of Lemma 4, p. 145]\label{1}
For $\alpha > 0, \beta \in \RR$, there exist positive constants $C_\alpha$ so that
\[ S_*^{\alpha+i\beta} f \leq C_\alpha e^{2 \beta^2} S_*^0|f| \]
holds pointwise.
\end{lemma}

\medskip

\begin{lemma}[cf. \cite{NS}, Proof of Lemma 5, pp. 145-146]\label{2}
For each nonpositive integer $-m \leq 0$, and each real $\beta$, there exist positive constants $C_m$ so that
\[ S^{-m+i \beta}_*f \leq C_m e^{3 \beta^2} \left( S_*^{-m-1}f + S_*^{-m}f + \dots + S_*^{-1} f \right)\]
holds pointwise.
\end{lemma}
Although the conclusion of this lemma holds, the proof offered in Nevo-Stein contains a small gap, due to an incorrect application of summation by parts \cite[first paragraph, p. 146]{NS}. We therefore provide a full proof -- which still follows the reasoning of \cite{NS} -- in the below appendix.

\medskip

\begin{lemma}[cf. \cite{NS}, Proof of Lemma 5, p. 147]\label{3}
Define
\[ R_m f(x)^2 := \sum_{0 \leq k \leq N/4} (k+1)^{2m-1}|S_k^{-m-1}f(x)|^2. \]
Then there exists a positive constant $c_{-m}$ so that
\[ S_*^{-m} f \leq c_{-m} R_mf  + 2S_*^{-(m-1)}f\]
holds pointwise.
\end{lemma}

Temporarily assuming the below proposition, let us see how the above Lemmas allow us to complete the proof.

\begin{proposition}\label{main}
With
\[ \aligned
R_mf(x)^2 &:= \sum_{0 \leq k \leq N/4} (k+1)^{2m-1}|S_k^{-m-1}f(x)|^2 \\
&= \sum_{0 \leq k \leq N/4} (k+1)^{2m-1} |\triangle^{m} \sigma_{2k}* f(x)|^2, \endaligned \]
there exists absolute constants $C'_m$ so that for each $N$
\[ \|R_m f \|_{L^2(\I^N)} \leq C'_m \| f\|_{L^2(\I^N)},\]
independent of $N$.
\end{proposition}

\begin{proof}[Proof of Theorem \ref{big}, Assuming Proposition \ref{main}]
By Proposition \ref{smooth},
we know that there exists absolute constants $\{A_p\}$, $1<p \leq \infty$, so that for each $N$,
\[ \| S_*^0 |f| \|_{L^p(\I^N)} \leq A_p \|f\|_{L^p(\I^n)},\]
where the operators $\{S_*^0\}$ are $N$-dependent, but the bounds are not.

By Lemma \ref{1}, for each $\alpha > 0, \beta \in \RR$, we therefore have the bound
\[ \| S_*^{\alpha+i\beta} f \|_{L^p(\I^N)} \leq C_\alpha e^{2 \beta^2} A_p \| f \|_{L^p(\I^N)},\]
valid for each $N$.

By Proposition \ref{main}, Lemma \ref{3}, and induction on $m$, we see that there exist absolute constants $\{B_2^m\}, m \geq 1$ so that for all $N$,
\[ \|S_*^{-m} f\|_{L^2(\I^N)} \leq B_2^m \|f\|_{L^2(\I^N)}.\]
By Lemma \ref{2}, this means that for all $N$, there exist absolute constants $D_2^m$ so that
\[ \|S_*^{-m+i\beta} f\|_{L^2(\I^N)} \leq e^{3\beta^2} D_2^m \|f\|_{L^2(\I^N)} \]
for all $N$.

To prove the theorem, we linearize our maximal function $\mathcal{M} f \equiv S_*^{-1} f(x)$: we let
\[ \mathcal{R} : \I^N \to [0,M] \text{ even} \] 
be a ``choice'' function satisfying
\[ S_*^{-1} f (x) = \sigma_{\mathcal{R}(x)} * f(x),\]
and define the linear operators
\[ S_{\mathcal{R}}^{\lambda}f(x) := (\mathcal{R}(x) + 1)^{- \lambda -1} S_{\mathcal{R}(x)}^\lambda f(x).\]
By analytic interpolation of operators as in \cite{S} or \cite[\S 5]{NS}, we may bound $S_{\mathcal{R}}^{-1} f$ in $L^p$ for each $p > 1$.
\end{proof}
It remains only to prove Proposition \ref{main}, which we accomplish in the following subsection.

\subsection{Proof of Proposition \ref{main}}
\begin{proof}
By Plancherel, it is enough to show that there exists an absolute constant, $C'_m$, independent of $N$, so that for all $0 \leq r \leq N$
\[ \sum_{k=0}^{M/2} (k+1)^{2m-1} |\F \triangle^{m} \sigma_{2k}|^2(r)
\leq C'_m. \]
We recall $M = \lceil N/2 \rceil$.

We need to show
\[ \sup_{r \leq N} \ \sum_{k=0}^{M/2} (k+1)^{2m-1} |\partial^{m} \kappa_r^N(2k)|^2 \leq C'_m.\]
Using the symmetry of the Krawtchouk polynoials,
\[ |\partial^m \kappa_r^N(2k)| = |\partial^m \kappa_{N-r}^N(2k)|,\]
it suffices only to prove
\[ \sup_{r \leq M} \ \sum_{k=0}^{M/2} (k+1)^{2m-1} |\partial^{m} \kappa_r^N(2k)|^2 \leq C'_m.\]

To do so, we can and will assume that $N$ is much larger than $m$ -- say $N \geq (10m)^{10m}$. This is since we are free to increase $C'_m$ finitely many times -- as long as the number of times we increase $C'_m$ is independent of $N$.

We will use the upper bound \eqref{e:smallr} valid for $m < k,r \leq M$,
obtained from Lemma \ref{KR}, but we first dispose of the boundary case $r=m$, in which case
\[ \kappa_{r-m}^{N-2m}(2(k-m)) = \kappa_{0}^{N-2m}(2(k-m)) = 1.\]
In this instance, we estimate
\[ \aligned
\sum_{k=0}^{M/2} (k+1)^{2m-1} |\partial^{m} \kappa_r^N(2k)|^2 &\leq
\sum_{k=1}^{N} k^{2m-1} \cdot \left( \frac{1}{\binom{N}{m}} \right)^2 \\
&\lesssim \left( \frac{N^m}{\binom{N}{m}} \right)^2 \\
&\leq \left( \frac{N}{N-m} \right)^{2m} \cdot (m!)^2 \\
&\lesssim_m 1, \endaligned\]
since $\frac{N}{N-m} \leq 2$ because $N$ is so much larger than $m$.

Henceforth, we may assume $r >m$, so that the estimate
\[ \aligned
|\partial^{m} \kappa_r^N(2k)| &\leq
(-4)^m \frac{\binom{N-2m}{r-m}}{\binom{N}{r}} e^{-2c \frac{r-m}{N-2m} \cdot (k-m)} \\
&\lesssim_m \frac{\binom{N-2m}{r-m}}{\binom{N}{r}} e^{-2c \frac{r-m}{N-2m} \cdot (k-m)} \endaligned\]
holds for $k > m$ (recall that $\partial^m \kappa_r(2k) = 0$ for $m \geq k$).
Indeed, we may bound
\[ \aligned
\sum_{k=0}^{M/2} (k+1)^{2m-1} |\partial^{m} \kappa_r^N(2k)|^2 &\leq
\sum_{k=m}^{\infty} (k+1)^{2m-1} |\partial^{m} \kappa_r^N(2k)|^2 \\
&\lesssim_m \sum_{k=m}^{\infty} (k+1)^{2m-1} \left| \frac{\binom{N-2m}{r-m}}{\binom{N}{r}} e^{-2c \frac{r-m}{N-2m} \cdot (k-m)} \right|^2 \\
&= \left( \frac{\binom{N-2m}{r-m}}{\binom{N}{r}} \right)^2 \sum_{k=m}^{\infty} (k+1)^{2m-1} e^{-4c \frac{r-m}{N-2m} \cdot (k-m)} \\
&= \left( \frac{\binom{N-2m}{r-m}}{\binom{N}{r}} \right)^2 \sum_{k=0}^{\infty} (k+(m+1))^{2m-1} e^{-4c \frac{r-m}{N-2m} \cdot k}. \endaligned\]

We record the following easy lemma concerning infinite series:
\begin{lemma}
For $|t|<1$, define the operator, $L$, acting on smooth $g : \RR \to \RR$ by
\[ Lg(t):= t \frac{dg}{dt}(t),\]
and let $L^n$ denote the $n$-fold composition, i.e. $L^ng= L(L^{n-1}g), \ n \geq 2$.
Then
\[ L^n \frac{1}{1-t} = \sum_{k=0}^{\infty} k^n t^k \]
can be expressed as $\frac{t^n + p_n(t)}{(1-t)^{n+1}},$
where $p_n(t) := \sum_{j <n} a_j^n t^j$
is a polynomial of degree $n-1$.

In particular, for $t < 1$, we may bound
\[ \left|\frac{t^n + p_n(t)}{(1-t)^{n+1}} \right| \leq \frac{A_n}{(1-t)^{n+1}},\]
where we let 
\begin{equation}\label{e:coeffmax}
A_n:= 1+\sum_{j<n} |a_j^n|.
\end{equation}
\end{lemma}
\begin{proof}[Proof of Lemma]
This follows by induction.
\end{proof}

Now, following the lead (and notation) of \cite[\S 4]{HKS}, we set
\begin{equation}\label{e:alpha}
\alpha = \alpha(r):= 4c' \cdot \frac{r-m}{N-2m};
\end{equation}
where $c' := \min\{ c, \frac{1}{100}\}$, so that
\[ |\alpha(r)| \leq 4c' \cdot \frac{r-m}{N-2m} \leq \frac{4c'}{2} < 1/50\]
for all $r$.

We now generously estimate
\[ \aligned
\sum_{k=0}^{\infty} (k+(m+1))^{2m-1} e^{-4c \frac{r-m}{N-2m} \cdot k} &\leq
\sum_{k=0}^{\infty} (k+(m+1))^{2m-1} e^{-\alpha \cdot k} \\
&= \sum_{k=0}^{\infty} \left( \sum_{j=0}^{2m-1} \binom{2m-1}{j} k^j (m+1)^{2m-1-j} \right) e^{-\alpha \cdot k} \\
&\leq (10m+10)^{10m} \cdot \sum_{k=0}^{\infty} \left( \sum_{j=0}^{2m-1} k^j \right) e^{-\alpha \cdot k} \\
&\lesssim_m \sum_{j=0}^{2m-1} \left( \sum_{k=0}^{\infty} k^j e^{-\alpha \cdot k} \right) \\
&\leq \sum_{j=0}^{2m-1} \frac{A_j}{(1-e^{-\alpha})^{j+1}} \\
&\lesssim \sum_{j=0}^{2m-1} \frac{A_j}{\alpha^{j+1}} \\
&\leq 2m \cdot A_m^* \cdot \alpha^{-2m} \\
&\lesssim_m \alpha^{-2m}, \endaligned \]
where we let 
\[ A_m^*:= \max_{n \leq m} A_n,\]
and the $\{ A_n \}$ are defined in \eqref{e:coeffmax}.
Note that we used that $\alpha < 1/50$ in estimating
\[ \frac{1}{1 - e^{-\alpha}} \lesssim \frac{1}{\alpha}.\]

The upshot is that we may bound
\[ \sum_{k=0}^{\infty} (k+(m+1))^{2m-1} e^{-4c \frac{r-m}{N-2m} \cdot k} \lesssim_m \left( \frac{N-2m}{r-m} \right)^{2m},\]
so that we have
\[
\sum_{k=0}^{M/2} (k+1)^{2m-1} |\partial^{m} \kappa_r^N(2k)|^2 \lesssim
\left( \left( \frac{\binom{N-2m}{r-m}}{\binom{N}{r}} \right) \left( \frac{N-2m}{r-m} \right)^{m} \right)^2.\]

The task is now to show that for all $N$, there exists an absolute $C_m'$ so that for all $m < r \leq N$
\begin{equation}\label{e:goal}
\left( \frac{\binom{N-2m}{r-m}}{\binom{N}{r}} \right) \left( \frac{N-2m}{r-m} \right)^{m} \leq C'_m.
\end{equation}
To do so, we observe that the domination of $m$ by $N$ allows us to approximate
\begin{equation}\label{e:comp} 
\frac{(N-2m)!}{N!} \approx_m N^{-2m}.
\end{equation}
Indeed, since $(10m)^{10m} \leq N$ implies that $(N-2m) \geq N/2$, we may bound the left hand side of \eqref{e:comp} by
\[ N^{-2m} \leq \frac{(N-2m)!}{N!} \leq 2^{2m} \cdot N^{-2m}.\]
Using this estimate, we re-organize the left side of \eqref{e:goal} and bound:
\[ \aligned
\left( \frac{\binom{N-2m}{r-m}}{\binom{N}{r}} \right) \left( \frac{N-2m}{r-m} \right)^{m} 
&= \frac{(N-2m)!}{N!} \cdot \frac{r!}{(r-m)!} \cdot \frac{(N-r)!}{(N-r-m)!} \cdot \frac{(N-2m)^{m}}{(r-m)^m} \\
&\lesssim_m \frac{1}{N^{2m}} \cdot \frac{r!}{(r-m)!} \cdot \frac{(N-r)!}{(N-r-m)!} \cdot \frac{(N-2m)^{m}}{(r-m)^m}. \endaligned \]

We now fix a large constant $K$, and consider the two regimes:
\begin{itemize}
\item $r \leq K m$; and
\item $K m < r < N/2$.
\end{itemize}

In the first regime, we have $\frac{r!}{(r-m)!} \lesssim_m 1$, and $\frac{(N-r)!}{(N-r-m)!} \leq (N-r)^m$, so we may bound
\[ \aligned
\frac{1}{N^{2m}} \cdot \frac{r!}{(r-m)!} \cdot \frac{(N-r)!}{(N-r-m)!} \cdot \frac{(N-2m)^{m}}{(r-m)^m} &\lesssim_m \frac{1}{N^{2m}} \cdot \frac{(N-r)!}{(N-r-m)!} \cdot \frac{(N-2m)^{m}}{(r-m)^m}\\
&\leq \frac{1}{N^{2m}} \cdot (N-r)^{m} \cdot (N-2m)^m \\
& \leq 1. \endaligned\]

In the second regime, arguing as in \eqref{e:comp}, we have
\[ \frac{r!}{(r-m)!} \approx_m r^m \approx_m (r-m)^m,\]
so that we may similarly bound
\[ \aligned
\frac{1}{N^{2m}} \cdot \frac{r!}{(r-m)!} \cdot \frac{(N-r)!}{(N-r-m)!} \cdot \frac{(N-2m)^{m}}{(r-m)^m} &\lesssim_m \frac{1}{N^{2m}} \cdot r^m \frac{(N-r)!}{(N-r-m)!} \cdot \frac{(N-2m)^{m}}{r^m}\\
&\leq \frac{1}{N^{2m}} \cdot (N-r)^{m} \cdot (N-2m)^m \\
& \leq 1, \endaligned\]
which yields \eqref{e:goal}, and completes the proof.
\end{proof}

\section{Failure of the Uniform Weak-Type Estimate}
We may assume $N$ is even. Testing $\delta = \mathbf{1}_{\{(0,\dots,0)\}}$, the point mass at the origin,
\[ |\{ (\delta)^{**} > 0 \}| = \left| \left\{ (\delta)^{**} \geq \binom{N}{N/2}^{-1} \right\} \right| = 2^N,\]
while $\binom{N}{N/2} \approx \frac{2^N}{\sqrt{N}}$ by Stirling's Formula, so
\[ \binom{N}{N/2}^{-1} \left| \left\{  (\delta)^{**} \geq \binom{N}{N/2}^{-1} \right\} \right| \gtrsim \sqrt{N} \to \infty \]
as $N \to \infty$.
Consequently, no uniform weak-type bound holds. In fact, this construction generalizes to more complicated product settings:

\begin{proposition}
For $k \geq 1$, we consider the groups $\Z_{k+1}^N$, equipped with the ``$l^0$'' metric,
\[ |y| = \left| \big( y(1), \dots, y(N) \big) \right| = \#\{ 1 \leq i \leq N : y(i) \neq 0\},\]
and counting measure.
With $\sigma_j = \frac{1}{| \{|y|=j \} |} 1_{ \{|y|=j \} }$ as above, we similarly define
\[ f^{**} (x):= \sup_j | \sigma_j *f|(x).\]
Then for any $k \geq 1$, the best constant, $C_k(N)$, satisfying the weak-type bound
\[ \| f^{**} \|_{L^{1,\infty}( \Z_{k+1}^N ) } \leq C_k(N) \| f\|_{L^1( \Z_{k+1}^N ) } \]
grows at least like $\sqrt{N}$.
\end{proposition}
\begin{proof} We may assume that $N$ is a multiple of $k+1$.

Noting that inside $\Z_{k+1}^N$
\[ | \{|y|=j \} | = k^j \binom{N}{j}, \]
we use Stirling's formula to approximate
\[ \max_j \{ | \{|y|=j \} | \} \approx \left| \left\{ |y|= \frac{k}{k+1} N \right\} \right| \approx_k \frac{1}{\sqrt{N}} (k+1)^N.\]
With $\delta$ as above (the point-mass at the origin), we similarly estimate
\[
\sqrt{N} (k+1)^{-N} |\{ (\delta)^{**} \gtrsim  \sqrt{N} (k+1)^{-N} \}| =
\sqrt{N} (k+1)^{-N} |\{ (\delta)^{**} >  0 \}| = \sqrt{N}.\]
\end{proof}

\begin{remark}
Further investigation of the spherical maximal function on the groups $\Z_{k+1}^N$ has recently been conducted in \cite{GKKS}; in particular, the spherical maximal function is shown to satisfy dimension-independent bounds on $L^p$, $p > 1$.
\end{remark}

\section{Appendix}
We here provide a full proof of Lemma \ref{2}, reproduced below for the reader's convenience.

\begin{lemma}[cf.\ \cite{NS}, Proof of Lemma 5, pp. 145-146]\label{22}
For each nonpositive integer $-m \leq 0$, and each real $\beta$, there exist positive constants $C_m$ so that
\[ S^{-m+i \beta}_*f \leq C_m e^{3 \beta^2} \left( S_*^{-m-1}f + S_*^{-m}f + \dots + S_*^{-1} f \right)\]
holds pointwise.
\end{lemma}

To do so, we will need to recall a few results from \cite[\S 2.2]{NS} concerning complex Cesaro sums.

\begin{lemma}[\cite{NS}, Lemma 2]\label{A}
For $\lambda, \delta \in \C$,
\begin{enumerate}
\item We have the following pointwise ``convolution'' identity:
\[ S_n^{\lambda + \delta}f(x) = \sum_{k=0}^n A_{n-k}^{\delta - 1} S_k^\lambda f(x);\]
\item $A_n^\lambda - A_{n-1}^\lambda = A_n^{\lambda -1}$; and
\item We have the following pointwise identity:
\[ S_n^\lambda f(x) - S_{n-1}^\lambda f(x) = S_n^{\lambda -1 }f(x).\]
\end{enumerate}
\end{lemma}

\begin{lemma}[\cite{NS}, Lemma 3]\label{B}
For $\alpha,\beta \in \RR$, let $\lambda = \alpha + i\beta$. Then:
\begin{enumerate}
\item For $\alpha > -1$, there exists $b_\alpha > 0$ so that
\[ b_\alpha^{-1} \leq \frac{A_n^\alpha}{(n+1)^\alpha} \leq b_\alpha ;\]
\item For $\alpha > -1$, there exists $a_\alpha > 0$ so that
\[ 1 \leq \left| \frac{A_n^{\alpha + i\beta}}{A_n^\alpha} \right|\leq a_\alpha e^{2\beta^2};\]
\item For $m \in \N$ there exists $B_m$ so that
\[ |(n+1)^m A_n^{-m + i\beta} | \leq B_m e^{3\beta^2}.\]
\end{enumerate}
\end{lemma}

We proceed to the proof:

We wish to show that for each $n \leq M/2$, we have the upper estimate:
\[ |(n+1)^{m-1} S_n^{-m+ i \beta}f| \lesssim_m \sum_{j=1}^{m+1} S_*^{-j}f.\]
To do so, we separate into two regimes:
\begin{enumerate}
\item When $n \leq 10m $; and
\item When $n > 10 m  $.
\end{enumerate}

\begin{proof}[Regime One] In regime one, we can afford to be somewhat crude:

Using the ``convolution'' formula from Lemma \ref{A}, we express
\[ S_n^{-m+i\beta} f = \sum_{k=0}^n A_{n-k}^{i\beta} S_k^{-m-1}f. \]
Using the estimate from Lemma \ref{B},
\[ |A_{n-k}^{i\beta}| \leq a_0 e^{2\beta^2},\]
and the estimate, valid for $0\leq k \leq n \lesssim m$,
\[ |S_k^{-m-1} f| \leq (k+1)^m S_*^{-m-1}f \lesssim (m+1)^m S_*^{-m-1}f,\]
we may bound
\[ \aligned
| (n+1)^{m-1} S_n^{-m+i\beta} f| &\leq (n+1)^{m-1} n \cdot \max_{0 \leq k \leq n} |A_{n-k}^{i\beta}| \cdot |S_k^{-m-1}f| \\
&\lesssim (m+1)^{m-1} \cdot m \cdot a_0 e^{2\beta^2} \cdot S_*^{-m-1}f \\
&\lesssim_m e^{2\beta^2} S_*^{-m-1}f,
\endaligned\]
as desired. Here, we trivially estimated $(m+1)^m \lesssim_m 1$.
\end{proof}

It remains to consider regime two, when $n > 10 m$; we will adhere closely to the proof of \cite[Lemma 5]{NS}. 

First, though, we isolate a useful summation-by-parts identity:
\begin{equation}\label{e:sbp}
\sum_{k=0}^t A_{n-k}^{-l+i\beta} S_k^{-m-1 + l} f = A_{n-t}^{-l + i \beta} S_t^{-m -1 + l}f + \sum_{k=0}^{t-1} A_{n-k}^{-l-1+i\beta} S_k^{-m+l}f.
\end{equation}
\begin{proof}[Proof of \eqref{e:sbp}]
To establish \eqref{e:sbp}, we sum the left hand side by parts to reduce matters to showing that
\begin{equation}\label{e:sbp2}
\sum_{r = 0}^{t-1} \left( A_{n-r}^{-l + i\beta} - A_{n-r-1}^{-l+i\beta} \right) \cdot \left( \sum_{p =0}^r S_p^{-m-1 + l} f \right) = \sum_{k=0}^{t-1} A_{n-k}^{-l-1+i\beta} S_k^{-m+l}f.
\end{equation}
We now apply Lemma \ref{A} (2), with $\lambda = - l + i \beta$, to simplify the left hand side of \eqref{e:sbp2},
\[ \sum_{r =0}^t A_{n-r}^{-l -1 + i\beta} \cdot \left( \sum_{p=0}^{ r} S_p^{-m-1 + l} f \right).\]
By Lemma \ref{A} (3), for each $p \leq r$, we may substitute
\[ S_p^{-m-1 +l}f = S_p^{-m+l} f - S_{p-1}^{-m+l}f,\]
(here used $\lambda = - m + l$) to telescope
\[ \sum_{p =0}^{r} S_p^{-m-1 + l} f = S_r^{-m + l} f, \]
taking into account \eqref{e:id} to ensure cancellation at the lower boundary, $p=0$.
\end{proof}
With this in hand, we turn to regime two, when $n > 10m$.
\begin{proof}[Regime Two]
Setting $n_2 := \lfloor n/2 \rfloor$, we again use the convolution formula from Lemma \ref{A} to decompose
\[ \aligned
S_n^{-m+ i\beta} f &= \sum_{k=0}^{n_2 + m } A_{n-k}^{i\beta} S_k^{-m-1}f + \sum_{k = n_2 + m + 1}^n A_{n-k}^{i\beta} S_k^{-m-1} f \\
&=: \Sigma_1 + \Sigma_2.
\endaligned \]
We may majorize the modulus of the second sum by
\[ \lesssim e^{2 \beta^2} \sum_{k = n_2 + m +1}^n (k+1)^{-m} S_*^{-m-1} f \lesssim
e^{2\beta^2} \cdot (n+1)^{1-m} S_*^{-m-1} f, \]
since Lemma \ref{B} allows us to estimate
\[ |A_{n-k}^{i\beta}| \leq a_0 e^{2 \beta^2},\]
and we have $n > 10m$.

We now focus on estimating $(n+1)^{m-1} \cdot \Sigma_1$, where we recall
\begin{equation}\label{e:sigma1}
\Sigma_1 := \sum_{k=0}^{n_2 + m } A_{n-k}^{i\beta} S_k^{-m-1}f.
\end{equation}
We will handle $\Sigma_1$ by repeated applications of identity \eqref{e:sbp}.

Indeed, we may express
\[ \aligned
\Sigma_1 &= A_{n-(n_2+m)}^{i\beta} S_{n_2+m}^{-m-1} f + \dots + A_{n - (n_2 + m -l)}^{-l + i\beta} S_{n_2+m-l}^{-m-1+l} f \\
& \qquad + \dots + A_{n - (n_2 + 1)}^{-(m-1) + i\beta} S_{n_2 +1}^{-2} f \\
& \qquad \qquad + \sum_{k=0}^{n_2} A_{n-k}^{-m + i\beta} S_k^{-1} f. \endaligned \]

Using Lemma \ref{B} and the definition of $S_*^{-t}$ for $1 \leq t \leq m$, we may bound the top term:
\[ \aligned 
| A_{n-(n_2+m)}^{i\beta} S_{n_2+m}^{-m-1} f | &\leq a_0 e^{2 \beta^2} (n_2 + m + 1)^{-m} S_*^{-m-1}f \\
&\lesssim e^{2\beta^2} (n+1)^{-m} S_*^{-m-1}f; \endaligned \]
the intermediate terms, $1 \leq l \leq m-1$:
\[ \aligned 
&|A_{n - (n_2 + m -l)}^{-l + i \beta} S_{n_2+m -l }^{-m -1 + l} f | \\
& \qquad \leq (n-(n_2+m-l) + 1)^{-l} \cdot B_l e^{3\beta^2} \cdot ( n_2 + m - l + 1)^{-m+l} S_*^{-m+l}f \\
& \qquad \qquad \lesssim_l e^{3\beta^2} (n+1)^{-m} S_*^{-m+l} f; \endaligned \]
and the sum
\[ \aligned 
&|\sum_{k=0}^{n_2} A_{n-k}^{-m + i\beta} S_k^{-1}f| \\
& \qquad \leq (n_2 + 1) \cdot \max_{0 \leq k \leq n_2} (n-k+1)^{-m} \cdot B_m e^{3\beta^2} S_*^{-1}f \\
& \qquad \qquad \lesssim_m e^{3\beta^2} (n+1)^{1-m} S_*^{-1} f. \endaligned\]

Summing and taking into account the normalizing factor of $(n+1)^{m-1}$ completes the proof.
\end{proof}


\begin{thebibliography}{9}

\bibitem{Al}
Aldaz, J.M. The weak type (1,1) bounds for the maximal function associated to cubes grow to infinity with the dimension.
Ann. of Math. Volume 173 (2) Pages 1013-1023 (2011).

\bibitem{Au}
Aubrun, G. Maximal inequality for high dimensional cubes, arXiv:0902.4305.

\bibitem{B}
Bourgain, J. On the Hardy-Littlewood maximal function for the cube.
http://arxiv.org/abs/1212.2661.

\bibitem{GKKS}
Greenblatt, J.; Kolla, A.; Krause, B.; Schulman, L. Dimension-Free $L^p$ Inequalities in $\Z_{m+1}^N$. Preprint, http://arxiv.org/pdf/1406.7229.pdf.


\bibitem{H}
Hatami, H. Harmonic Analysis of Boolean Functions. Lecture Notes 9-11. http://cs.mcgill.ca/~hatami/comp760-2011/

\bibitem{HKS}
Aram W. Harrow, Alexandra Kolla, Leonard J. Schulman. Dimension-free $L^2$ Maximal Inequality for Spherical Means in
the Hypercube. http://arxiv.org/pdf/1209.4148v1.pdf

\bibitem{K}
I. Krasikov. Nonnegative quadratic forms and bounds on orthogonal polynomials. Journal of
Approximation Theory, 111(1):31�49, 2001.

\bibitem{L}
V. I. Levenshtein. Krawtchouk polynomials and universal bounds for codes and designs in Hamming spaces. IEEE Trans. Inf. Th., 41(5):1303�1321, Sep 1995.

\bibitem{MPS}
Menarguez, T., Perez, S., Soria, F.
The Mehler Maximal Function: A Geometric Proof of the Weak Type 1.
J. London Math. Soc. (2) 61 (2000), no. 3, 846�856.1469-7750


\bibitem{NS}
A. Nevo and E. M. Stein. A generalization of Birkhoff�s pointwise ergodic theorem. Acta Mathematica, 173:135�154, 1994.

\bibitem{NT}
Naor, A., Tao, T. Random martingales and localization of maximal inequalities. J. Funct. Anal. 259 (2010), no. 3, 731�779.

\bibitem{O}
Ornstein, D. On the pointwise behavior of iterates ofa self-adjoint operator. J. Math. Mech. 18 (1968/1969), 473-477.

\bibitem{RT}
Rochberg, R; Taibleson, M. An averaging operator on a tree. Harmonic analysis and partial differential equations (El Escorial, 1987), 207�213, Lecture Notes in Math., 1384, Springer, Berlin, 1989

\bibitem{S}
E. M. Stein. On the maximal ergodic theorem. Proc Natl Acad Sci USA, 47(12):1894�1897,
December 1961.

\bibitem{S1}
E. M. Stein. Harmonic analysis: real-variable methods, orthogonality, and oscillatory integrals. Princeton Mathematical Series, 43. Monographs in Harmonic Analysis, III. Princeton University Press, Princeton, NJ, 1993.

\bibitem{SJ1}
Sj\"{o}gren, Peter. On the maximal function for the Mehler kernel. Harmonic analysis (Cortona, 1982), 73�82,
Lecture Notes in Math., 992, Springer, Berlin, 1983.


\bibitem{SJ}
Operators associated with the Hermite semigroup - a survey. Proceedings of the conference dedicated to Professor Miguel de Guzm�n (El Escorial, 1996). J. Fourier Anal. Appl. 3 (1997), Special Issue, 813�823

\bibitem{Z}
Zygmund, A. On a theorem of Marcinkiewicz concerning interpolation of operations. J. Math.
Pures Appl., 9(35):223–248, 1956.

\end{thebibliography}
\end{document}